%% file: transversalKnotComplementProblem1.tex
\documentclass{amsart} 

\usepackage{amsmath,amssymb,amsthm} 
\usepackage{graphicx} 
\usepackage[arrow, matrix, curve]{xy} 
\usepackage{color} 
\usepackage{hyperref} 

\DeclareMathOperator{\tb}{tb}
\DeclareMathOperator{\rot}{rot}
\DeclareMathOperator{\lk}{lk}
\DeclareMathOperator{\slf}{sl}

\newcommand{\R}{\mathbb{R}}
\newcommand{\Z}{\mathbb{Z}}

\newcommand{\N}{\mathbb{N}}

\newcommand{\xist}{\xi_{\mathrm{st}}}
\newcommand{\xiot}{\xi_{\mathrm{ot}}}

\newtheoremstyle{thm}{}{}{\itshape}{}{\bfseries}{\hfill\\}{ }{}
\newtheoremstyle{definition}{}{}{}{}{\bfseries}{\hfill\\}{ }{}

\theoremstyle{thm}
\newtheorem{Theorem}{Theorem}[section]
\newtheorem{thm}[Theorem]{Theorem}
\newtheorem{lem}[Theorem]{Lemma}

\newtheorem{cor}[Theorem]{Corollary}
\newtheorem*{Theorem-ohne}{Theorem}

\newtheorem{ques}[Theorem]{Question}

\theoremstyle{definition}
\newtheorem{defi}[Theorem]{Definition}
\newtheorem{rem}[Theorem]{Remark}
\newtheorem{ex}[Theorem]{Example}

\begingroup\expandafter\expandafter\expandafter\endgroup
\expandafter\ifx\csname pdfsuppresswarningpagegroup\endcsname\relax
\else
\fi

\begin{document}


\title[The knot complement problem for transverse knots]{Cosmetic contact surgeries along transverse knots and the knot complement problem}

\author{Marc Kegel}

\address{Institut f\"ur Mathematik, Humboldt-Universit\"at zu Berlin,
Unter den Linden 6, 10099 Berlin, Germany}
\email{kegemarc@math.hu-berlin.de}

\date{\today}


\begin{abstract}
We study cosmetic contact surgeries along transverse knots in the standard contact $3$-sphere, i.e. contact surgeries that yield again the standard contact $3$-sphere. The main result is that we can exclude non-trivial cosmetic contact surgeries along all transverse knots not isotopic to the transverse unknot with self-linking number $-1$. 

As a corollary it follows that every transverse knot in the standard contact $3$-sphere is determined by the contactomorphism type of its exteriors. Moreover, we give counterexamples to this for transverse links in the standard contact $3$-sphere.
\end{abstract}

\date{\today} 



\maketitle


\section{Introduction}

\textit{Dehn surgery} is the process of removing a tubular neighborhood of a (tame) knot in the $3$-sphere $S^3$ and gluing a solid torus back by a homeomorphism of the boundaries to get a new $3$-manifold. This was first used in 1910 by Dehn for easy and effective constructions of homology spheres~\cite{De10}. In fact, a fundamental theorem due to Lickorish and Wallace says that every connected closed oriented $3$-manifold can be obtained by Dehn surgery along a \textbf{link} from $S^3$~\cite{Li62,Wa60}.

This connection between knot theory in $S^3$ and $3$-manifold topology in general lead to many interesting results in both areas. One of the most striking applications of Dehn surgery to knot theory is the \textit{knot complement theorem} for knots in $S^3$. In 1908 Tietze asked if a knot in $S^3$ is determined by the homeomorphism type of its complement~\cite{Ti08}. While this is not true for links in $S^3$~\cite{Wh37} and for knots in general $3$-manifolds~\cite{Ma92,Ro93}, the original question remained a long time open, until it was spectacularly answered in the affirmative by Gordon and Luecke in 1989~\cite{GoLu89}.

Gordon and Luecke proved a slightly more general statement about Dehn surgery, namely that if a non-trivial Dehn surgery along a single knot $K$ in $S^3$ gives back $S^3$ (a so-called \textit{cosmetic} Dehn surgery), then the knot $K$ has to be equivalent to the unknot. From this it follows very easily that every knot in $S^3$ is determined by its exterior and, by work of Edwards~\cite{Ed64}, also by its complement.

The connection between these two different results is as follows. Given the meridian on the boundary of a knot exterior, it is easy to reconstruct the knot. First one can glue a $2$-disk $D^2$ to the meridian in a unique way. Then the boundary of the resulting object is a $2$-sphere $S^2$ to which one can glue a $3$-ball $B^3$ in a unique way. Therefore there is a unique way to glue a solid torus to the boundary of a knot exterior, such that the meridian of this solid torus maps to the meridian of the knot. The knot is then recovered as the spine of this newly glued-in solid torus. With this discussion the knot complement problem reduces to the question, if the meridian on a knot exterior is uniquely determined. Or asked in other words, in how many different ways one can glue a solid torus to the boundary of a knot exterior to get $S^3$ back, which is nothing but the existence question of cosmetic Dehn surgeries.

Here we want to study similar questions for transverse knots in contact $3$-manifolds. A (coorientable) \textit{contact structure} on a smooth $3$-manifold $M$ is a \textit{completely non-integrable} $2$-plane field $\xi\subset TM$, meaning that there exists a globally defined $1$-form $\alpha$ on $M$, such that $\alpha\wedge d\alpha$ is a volume form of $M$. The so-called standard contact structure $\xist$ on $S^3\subset \R^4$ is given in Cartesian coordinates as $\xist=\ker(x_1\,dy_1-y_1\,dx_1+x_2\,dy_2-y_2\,dx_2)$. A smooth knot $T$ in a contact manifold $(M,\xi)$ is called \textit{transverse}, if it is always transverse to the contact structure. It is easy to see that every topological knot can be $C^0$-close approximated by a transverse knot and that transverse knots look locally the same (in a so called \textit{standard} tubular neighborhood). For this and other background on contact topology we refer the reader to the standard textbook~\cite{Ge08}.

By using that every transverse knot looks locally the same, Martinet could make Dehn surgery along transverse knots compatible with the contact structures~\cite[Theorem~4.1.2]{Ge08}. Together with the theorem of Lickorish and Wallace it follows that every closed oriented $3$-manifold carries a contact structure. In fact, every contact structure on any closed $3$-manifold can be obtained by contact surgery along a transverse link from $(S^3,\xist)$~\cite{Ga99,Ga02,DiGe04,BaEt13,Co15}.

However, the notion of contact Dehn surgery that we use here is slightly more general, in the sense that we remove a standard tubular neighborhood of a transverse knot and then glue back some arbitrary contact solid torus, rather than a preferred on.

In this article, we study cosmetic contact Dehn surgeries along transverse knots in $(S^3,\xist)$, i.e. contact Dehn surgeries yielding again $(S^3,\xist)$. From a result of Etnyre and Ghrist~\cite[Theorem 2.2]{EtGh99} it follows, that cosmetic contact surgeries do exist along a transverse unknot with self-linking number $\slf=-1$. In contrast to this, is our main result.

\begin{thm} [Cosmetic contact surgeries along stabilized transverse unknots] \label{thm:transverseSurgery}
Any contact Dehn surgery with topological framing $\pm 1/n$ along a stabilized transverse unknot yields an overtwisted $3$-sphere. 	
\end{thm}

Using the topological result of Gordon and Luecke mentioned above it is easy to deduce from Theorem~\ref{thm:transverseSurgery} the following generalizations of the results by Gordon and Luecke.

\begin{cor} [Transverse contact surgery theorem]\label{thm:transverseSurgery2}
Let $T$ be a transverse knot in $(S^3,\xist)$ not isotopic to a transverse unknot with $\slf=-1$. Then there exists no non-trivial contact Dehn surgery along $T$ yielding $(S^3,\xist)$.
\end{cor}

As in the topological category, we see that two isotopic transverse knots have contactomorphic \textit{exteriors}, i.e.\ complements of open standard tubular neighborhoods. The reverse implication is non-trivial and can be answered for transverse knots in $(S^3,\xist)$ using Corollary~\ref{thm:transverseSurgery2}.

\begin{cor} [Transverse knot exterior theorem] \label{cor:transverse} 
Two transverse knots in $(S^3,\xist)$ with contactomorphic exteriors are isotopic.
\end{cor}

In Section~\ref{section:prelim} we provide the necessary background about transverse knots in contact $3$-manifolds and contact Dehn surgery along transverse knots. Theorem~\ref{thm:transverseSurgery} and its corollaries are proven in Section~\ref{section:proofDehnSurgery} and~\ref{section:proofKnotExterior}. We conclude in Section~\ref{section:open} with explicit counterexamples to Corollary~\ref{cor:transverse} in the case of transverse links in $(S^3,\xist)$ and state a few related open problems.

For Legendrian knots Corollary~\ref{cor:transverse} is also mentioned in~\cite[Theorem~2.13]{Et05}, see also~\cite{Ke16}.

\begin{rem} [Direct proof via convex surface theory]
I was informed by the referee that Corollaries~\ref{thm:transverseSurgery2} and~\ref{cor:transverse} can be proven directly with known techniques due to Honda~\cite{Ho00}. In this form Corollaries~\ref{thm:transverseSurgery2} and~\ref{cor:transverse} has been known to some experts for some time.
\end{rem}

\subsection*{Acknowledgment.}
I would like to thank my advisor Hansj\"org Geiges, my colleagues Sebastian Durst and Christian Evers for useful discussions and comments on an earlier version of this article. Further I would like to thank Kai Zehmisch for his  explanations about the size of standard neighborhoods of transverse knots and the anonymous referee for explaining how to remove the conditions on the size of these neighborhoods from the results and pointing out a mistake in Section~\ref{section:open} in an earlier version.

 The results presented here are part of the authors thesis~\cite{Ke17}, which was partially supported by the \textit{DFG-Graduiertenkolleg 1269 Globale Strukturen in Geometrie und Analysis} at the Universit\"at zu K\"oln. The research of the author is now supported by the \textit{Berlin Mathematical School}.


\section{Preliminaries}
\label{section:prelim} 

In this section we recall the necessary background about transverse knots and contact Dehn surgery along them.

\subsection{Transverse knots in contact 3-manifolds}\hfill

A knot $T$ in a contact $3$-manifold $(M,\xi)$ with cooriented contact structure is called \textbf{transverse}, if the contact planes $\xi$ are everywhere transverse to the tangent space of the knot. Transverse knots admit a preferred orientation by the requirement that they are positive transverse to the contact structure. Here all transverse knots are understood to be oriented like this, but all results holds similar also for negative transverse knots. An important fact about transverse knots is that they locally all look the same~\cite[Example~2.5.16]{Ge08}.

\begin{lem} [Neighborhoods of transverse knots]  
Let $T$ be a transverse knot in $(M,\xi)$. Then there exists a \textbf{standard tubular neighborhood} $\nu T$ of $T$ in $M$ such that $(\nu T,\xi)$ is contactomorphic to 
\begin{equation*}
\big(S^1\times D^2_\varepsilon , \ker(d\theta+r^2\,d\varphi)\big),
\end{equation*} 
where $\theta$ is an angular coordinate on $S^1$ and $(r,\varphi)$ are polar coordinates on the disk $D^2_\varepsilon$ with radius $\varepsilon$ and the contactomorphism sends $T$ to $S^1\times \{0\}$.
\end{lem}

\begin{rem}[Size of the standard neighborhoods]
We can even assume the above neighborhood $(\mathring{\nu T},\xi)$ to be contactomorphic to $(S^1\times \R^2, \ker(d\theta+r^2\,d\varphi))$~\cite{El91}. For that consider, for $N\in\N$, the map
\begin{align*}
F_N\colon\big(S^1\times \R^2 , \ker(d\theta+r^2\,d\varphi)\big)&\longrightarrow\big(S^1\times D^2_{1/\sqrt{N}} , \ker(d\theta+r^2\,d\varphi)\big)\\
\big(\theta,r,\varphi\big)&\longmapsto\big(\theta,\rho(r)r,\varphi-N\theta\big),
\end{align*}
with 
\begin{equation*}
\rho(r)=\frac{1}{\sqrt{1+Nr^2}},
\end{equation*} 
which can be easily computed to be a contactomorphism. 

However, the map $F_N$ is topologically a $(-N)$-fold Dehn-twist along the meridian of $T$. Therefore, the framings of the two neighborhoods differ. It is possible to define invariants of transverse knots (or even topological knots) coming from the size of their tubular neighborhoods by looking at the possible values of the slope of the characteristic foliation of $\partial(\nu T)$ (for a fixed choice of framing of $T$), as explained for example in~\cite{Ga99,Ga02,EtHo05,BaEt13,Co15}. 

In the present article we will come across a similar notion of the size of a standard tubular neighborhood $\nu T$ of a transverse knot $T$, see Section~\ref{section:proofDehnSurgery}.
\end{rem}

\begin{defi}[Complements and exteriors]  
Let $T$ be a transverse knot in $(M,\xi)$. Then one calls $(M\setminus T,\xi)$ the \textbf{complement} of $T$ and $(M\setminus \mathring{\nu T},\xi)$ an \textbf{exterior} of $T$, where $\nu T$ denotes, as always, a standard tubular neighborhood of $T$.
\end{defi}

Here I want to consider transverse knots up to coarse equivalence.

\begin{defi}[Coarse equivalence]  
Let $T_1$ and $T_2$ be two transverse knots in $(M,\xi)$. Then $T_1$ is \textbf{(coarse) equivalent} to $T_2$, if there exists a contactomorphism $f$ of $(M,\xi)$
\begin{align*}
	f\colon (M,\xi)&\longrightarrow (M,\xi)\\
	T_1 &\longmapsto T_2
\end{align*}
that maps $T_1$ to $T_2$.
\end{defi}

\begin{rem}[Coarse equivalence vs transverse isotopy]
The (coarse) equivalence is in general a weaker condition than the equivalence given by transverse isotopy, since not every contactomorphism has to be isotopic to the identity. For example in overtwisted contact structures on $S^3$ this is in general not the case~\cite{Vo16} (compare also the discussion in Section~4.3 in~\cite{ElFr09}). But it is known that in $(S^3,\xi_{st})$ this two concepts are the same, since every contactomorphism of $(S^3,\xi_{st})$ is isotopic to the identity~\cite{El92}.
\end{rem}

If two transverse knots are coarse equivalent, their complements and exteriors are contactomorphic. The transverse knot complement or exterior problem asks if the reverse also holds. In Section~\ref{section:proofKnotExterior} we will prove Corollary~\ref{cor:transverse} saying that this is true for transverse knots in $(S^3,\xist)$. On the other hand, we will give counterexamples for transverse links in $(S^3,\xist)$ in Section~\ref{section:open}.

\subsection{Contact Dehn surgery along transverse knots}\hfill

Next, I recall the definition of contact Dehn surgery along transverse knots. Roughly speaking one cuts out the neighborhood of a transverse knot and glues a contact solid torus back in a different way to obtain a new contact $3$-manifold. More precisely:

\begin{defi}[Contact Dehn surgery along transverse knots]  \label{defi:cs}
Let $T$ be a transverse knot in $(M,\xi)$. Take a standard tubular neighborhood $\nu T$ and a non-trivial simple closed curve $r$ on $\partial (\nu T)$ and a diffeomorphism $\varphi$, such that
\[\begin{array}{ccc}
\varphi\colon \partial(S^1\times D^2)&\longrightarrow& \partial (\nu T)\\
	\{\text{pt}\}\times\partial D^2 :=\mu_0 &\longmapsto& r.
\end{array}\]
Then define
\[\begin{array}{rccccl}
\big(M_T(r),\xi_T(r)\big)&:=& \big(S^1\times D^2,\xi_S\big) &+& \big(M\setminus\mathring{\nu T},\xi\big)&\big/_\sim ,\\
&&\partial(S^1 \times D^2)\ni p&\sim& \varphi(p)\in \partial(\nu T),&
\end{array}\]
where the contact structure $\xi_S$ on $S^1\times D^2$ is chosen, such that $\xi_S$ and $\xi$ fit together to a global new contact structure $\xi_T(r)$ on $M_T(r)$.
One says that $(M_T(r),\xi_T(r))$ is obtained out of $M$ by \textbf{contact Dehn surgery} along $T$ with \textbf{slope}~$r$.
\end{defi}

It is easy to show that $M_T(r)$ is again a $3$-manifold independent of the choice of $\varphi$. Also it is a standard fact that every nullhomologous knot $T$ (for example if $M$ is $S^3$) has a \textbf{Seifert surface}, i.e. a compact oriented surface $S$ with oriented boundary $T$. Then one can obtain a special parallel copy of $T$ on $\partial (\nu T)$, called the \textbf{surface longitude}, by pushing $T$ in a Seifert surface $S$. (This is independent of the choice of the Seifert surface.) Then one can write any non-trivial simple closed curve $r$ uniquely as $r=p\mu+q\lambda$ with $\mu$ the meridian of $\nu T$ and $p$ and $q$ coprime integers. Therefore, it is often easier to think of $r$ instead of a simple closed curve as a rational number $r=p/q\in\mathbb{Q}\cup\{\infty\}$, where $\infty$ means $p=1$ and $q=0$. This rational number is called \textbf{surgery coefficient}.

Also it is not hard to show that there exists always a contact structure $\xi_S$ on $S^1 \times D^2$ that fits together with the old contact structure $\xi$ (see for example the proof of Theorem 4.1.2 in \cite{Ge08}), but this contact structure is in general not unique. Therefore, the resulting contact structure $\xi_T(r)$ on the new manifold is also not unique and depends moreover highly on the chosen standard tubular neighborhood $\nu T$. As already remarked in the introduction this notion of contact Dehn surgery is more general than the usual one, since we allow gluing back arbitrary contact solid tori.

Here we are interested in so-called \textbf{cosmetic contact Dehn surgeries} along transverse knots, i.e. situations where $(M_T(r),\xi_T(r))$ is again contactomorphic to $(M,\xi)$. That cosmetic contact Dehn surgeries do exist follows from the following theorem due to Etnyre and Ghrist~\cite[Theorem 2.2]{EtGh99}.

\begin{thm} [Existence of cosmetic contact surgeries]\label{thm:etnyre-ghrist}
There exists a transverse unknot $T$ with self-linking number $\slf=-1$ in $(S^3,\xist)$, such that for every surgery coefficient $r\in\mathbb{Q}\cup\{\infty\}$ there exists a contact Dehn surgery along $T$ yielding a tight contact manifold. 
\end{thm}

Since Dehn surgery along an unknot with surgery coefficient of the form $1/n$, for $n\in\Z$, gives back $S^3$ and since on $S^3$ the standard contact structure $\xist$ is the only tight contact structure, it follows that there exist cosmetic contact surgeries from $(S^3,\xist)$ to itself for every surgery coefficient of the form $1/n$.

Observe also that transverse unknots in tight contact $3$-manifolds are completely classified by their self-linking numbers~\cite[Theorem 5.1.1]{El93}. Moreover, the possible values for the self-linking numbers are all negative odd integers. In particular, there is only one transverse unknot with self-linking number $\slf=-1$. 

In contrast to the result of Etnyre and Ghrist, is our main result, Theorem~\ref{thm:transverseSurgery}, saying that there exists no non-trivial cosmetic contact Dehn surgeries along all other transverse unknots.

Notice also, that the contact surgeries in Theorem~\ref{thm:etnyre-ghrist} are done with respect to a sufficiently big tubular neighborhood of the knot $T$ to ensure tightness of the resulting manifold as explained in~\cite[Remark~2.1]{EtGh99}. For small tubular neighborhoods Theorem~\ref{thm:etnyre-ghrist} is not true. This I will explain in more detail in Section~\ref{section:proofDehnSurgery}.

\subsection{Admissible and inadmissible transverse surgery}\hfill

As mentioned before, the result of contact Dehn surgery along a transverse knot is not unique. It depends on the chosen tubular neighborhood of the transverse knot and the extension of the contact structure over the glued-in solid torus. 

There exist also the notions of admissible and inadmissible transverse surgery which give a natural choice of contact structure on the surgered manifold (depending only on the tubular neighborhood in the case of admissible transverse surgery), see~\cite{Ga99,Ga02,BaEt13,Co15} for the precise definitions. From the work of Conway~\cite[Theorem~1.10]{Co15} it follows that all inadmissible transverse surgeries along transverse unknots in $(S^3,\xist)$ with self-linking number $\slf\leq-2$ yield overtwisted contact manifolds. 

This notions of contact Dehn surgery along transverse knots are more useful when studying properties of contact structures on fixed manifolds. However, to deduce that a transverse knot is determined by its exterior as in Section~\ref{section:proofKnotExterior} we need to consider all possible extensions of the contact structure over the glued-in solid torus.


\section{Proof of Theorem~\ref{thm:transverseSurgery}}
\label{section:proofDehnSurgery} 

In~\cite{Ke16} a similar theorem is proven for \textbf{Legendrian} knots, i.e. smooth knots always tangent to the contact structure. The proof here follows the same ideas.

We denote by $U_{\slf}$ the transverse unknot with self-linking number $\slf$. By Eliashberg's classification of transverse unknots~\cite[Theorem 5.1.1]{El93}, these transverse knots are unique and provide a complete list of transverse unknots.

The idea of the proof is now to consider, for every transverse unknot $U_{\slf}$ (other than $U_{-1}$) and every surgery coefficient of the form $r=1/n$, a Legendrian knot $L$ in the exterior $(S^3\setminus \mathring{\nu U_{\slf}},\xist)$ of $U_{\slf}$. This Legendrian knot can also be seen as a Legendrian knot in the new surgered contact manifold $(S^3_{U_{\slf}}(r),\xi_{U_{\slf}}(r))$. By Rolfsen's formula for the change of framings under a Rolfsen twist~\cite[Section~9.H]{Ro76}, the new \textbf{Thurston--Bennequin invariant} $\tb_{new}$ of the Legendrian knot $L$ in the new contact manifold can be computed as 
\begin{equation*}
\tb_{new}=\tb_{old}-n\lk^2(L,T),
\end{equation*} 
where $\tb_{old}$ is the old Thurston--Bennequin number of $L$ in $(S^3,\xist)$ and $\lk(L,T)$ denotes the linking-number between $L$ and $T$ in $(S^3,\xist)$ (see also~\cite{Ke16} for a generalization of these formula to general situations). Although the new contact structure $\xi_{T}(r)$ is not unique, the new Thurston--Bennequin invariant $\tb_{new}$ is.

A theorem by Eliashberg~\cite{El93} states that a contact manifold is \textbf{tight (not overtwisted)}, if and only if every nullhomologous Legendrian knot fulfills the \textbf{Bennequin inequality} $\tb\leq 2g-1$, where $g$ denotes the genus of the knot. The strategy is then to find a Legendrian knot that violates the Bennequin inequality in the new surgered manifold, which then cannot be contactomorphic to $(S^3,\xist)$.

We want to present these situations in diagrams, called front projections. Observe that $(S^3\setminus\{p\},\xist)$ is contactomorphic to $(\R^3,\ker x\,dy+dz)$. Therefore, one can present knots in $(S^3,\xist)$ in their \textbf{front projection} $(x,y,z)\mapsto(y,z)$. If one parametrizes a knot as $(x(t),y(t),z(t))$ then the contact condition tells us that the knot is Legendrian if and only if $\dot z+x\dot y=0$. And similar such a knot is transverse if and only if $\dot z+x\dot y>0$. This conditions give restrictions on the behavior of these knots in crossings and vertical tangencies in front projections~\cite[Chapter~3]{Ge08}.

For example in Figure~\ref{fig:classificationOfTransverseUnknots} the transverse unknots $U_{\slf}$ are pictured in their front-projections and in Figure \ref{fig:crossings} one can see the 4 possibilities of intersections of a Legendrian knot $L$ and a transverse knot $T$ in the front-projection. From the conditions before it follows immediately that the intersection of type $(2)$ and $(4)$ is uniquely determined, but in type $(1)$ and $(3)$ in general both is possible. But here it is enough to know the behavior in type $(2)$ and $(4)$ (see also~\cite{DuKe16} and~\cite{Ke17}).
\begin{figure}[htbp] 
\centering
\def\svgwidth{0,96\columnwidth}
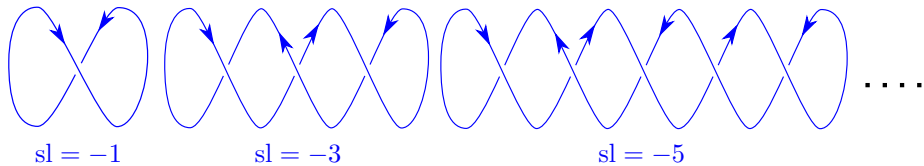
\caption{A complete list of transverse unknots in their front projections}
\label{fig:classificationOfTransverseUnknots}
\end{figure}	
\begin{figure}[htbp] 
\centering
\def\svgwidth{0,86\columnwidth}
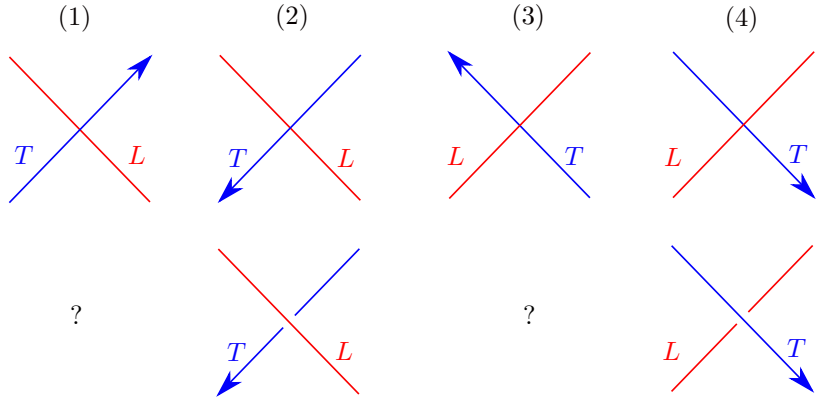
\caption{Front projections of Legendrian and transverse knots.}
\label{fig:crossings}
\end{figure}	

Consider now the front projection in Figure~\ref{fig:nsmallerzero} of a transverse unknot $T$ of type $U_{\slf}$ and a Legendrian unknot $L$ in the exterior of $T$. 
\begin{figure}[htbp] 
\centering
\def\svgwidth{0,86\columnwidth}
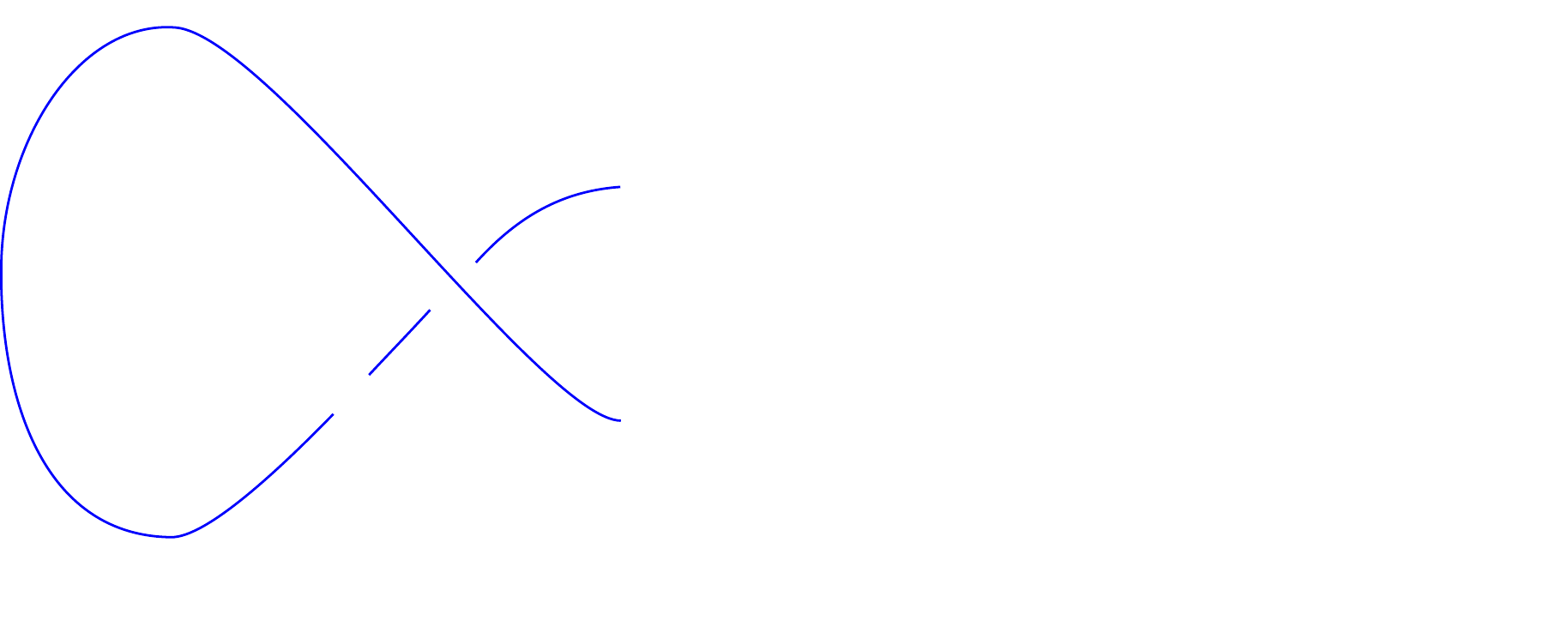
\caption{An overtwisted $3$-sphere for $n<0$.}
\label{fig:nsmallerzero}
\end{figure}	

If one does a $(1/n)$-surgery along $U_{\slf}$ then the knot $L$ is (topologically) again an unknot in some new contact $S^3$ (this can be seen by doing a \textbf{Rolfsen twist} along $U_{\slf}$, i.e. twisting $n$-times along the Seifert disk of $U_{\slf}$). If the resulting manifold is again $(S^3,\xi_{st})$ then the Bennequin inequality holds, i.e. $\tb_{new}\leq-1$. On the other hand we compute $\tb_{new}=-1-n$ with the above formula. So it follows that for $n<0$ the resulting contact $3$-sphere is not the standard one.
\begin{figure}[htbp] 
\centering
\def\svgwidth{0,85\columnwidth}
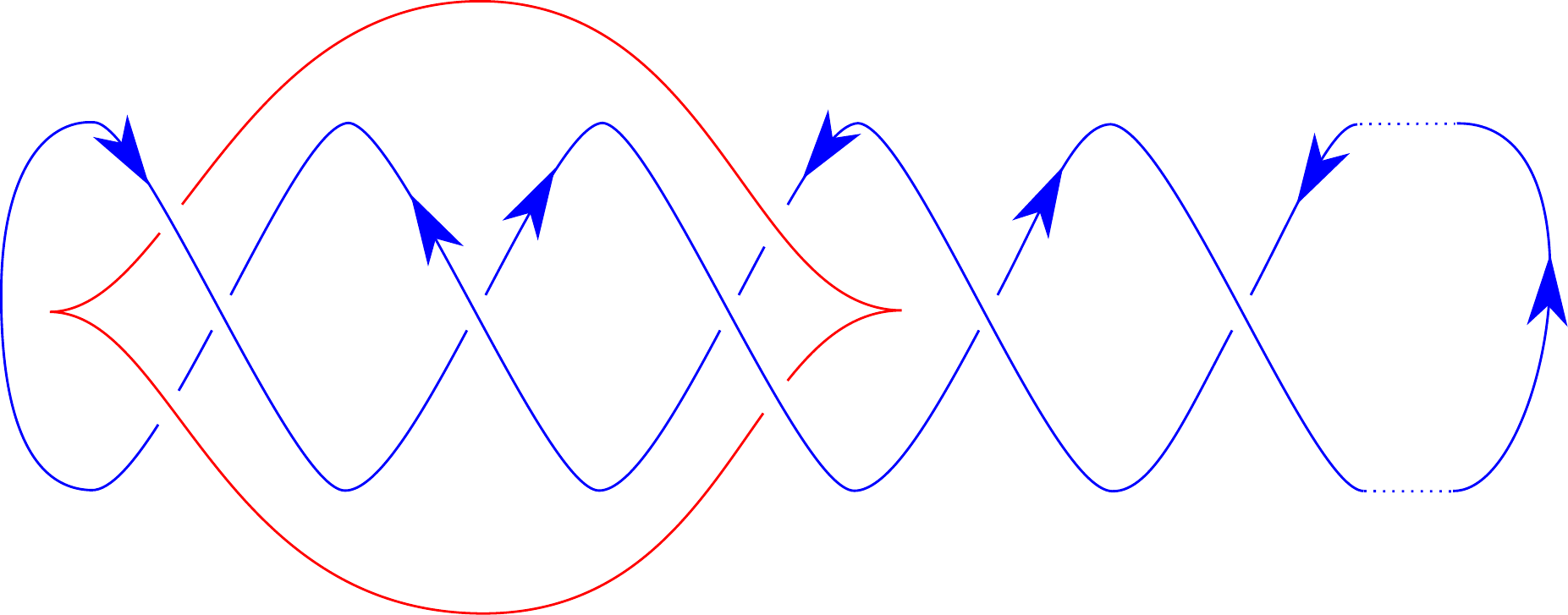
\caption{An overtwisted $3$-sphere for $n>0$.}
\label{fig:nbiggerzero}
\end{figure}	

For $n>0$ one considers the diagram from Figure~\ref{fig:nbiggerzero}. Observe that this does not work for the transverse unknot with self-linking number~$-1$, but for all other transverse unknots. The linking number is $\lk(T,L')=0$ and therefore the Thurston--Bennequin number of $L'$ stays the same, i.e. $\tb_{new}=\tb_{old}=-1$. 

To show that the resulting contact $3$-sphere is always overtwisted we first determine the new knot type $L'_{n}$ of $L'$ in the new $S^3$. Therefore, observe that topologically the link $T\sqcup L'$ is the Whitehead link, which can also be pictured like in Figure~\ref{fig:newlink}(i). By doing a Rolfsen twist along $T$ we get the new knot $L'_n$ in the new $S^3$ pictured in Figure~\ref{fig:newlink}(ii). 
\begin{figure}[htbp] 
\centering
\def\svgwidth{0,81\columnwidth}
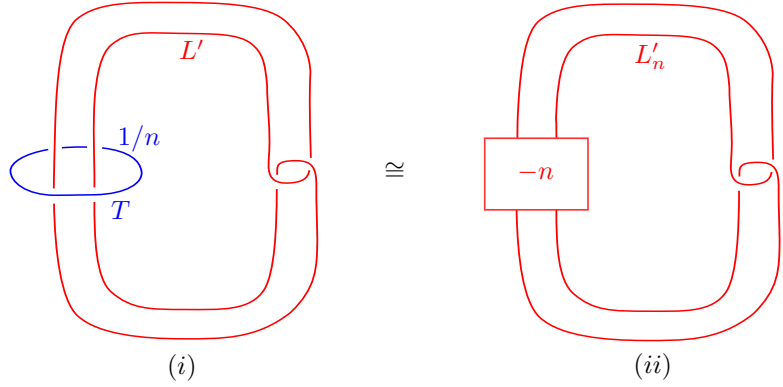
\caption{A Rolfsen twist along $T$ yields $L'_n$. Here the box represents $n$ negative full twists of the two parallel strands.}
\label{fig:newlink}
\end{figure}	

Of course here the Bennequin inequality cannot help because the minimal bound in the Bennequin inequality is $-1$. Therefore, one has to use a finer bound for the Thurston--Bennequin number. One such bound is the so-called \textbf{Kauffman bound}~\cite{Ru90,Fe02}, which says that for a Legendrian knot $L'$ in $(S^3,\xi_{st})$ the following inequality holds
\begin{equation*}
\tb(L')\leq\operatorname{min}\big\{\operatorname{deg}_x\big(F_{L'}(x,y)\big)\big\}-1,
\end{equation*} 
where $F_{L'}(x,y)$ is the Kauffman polynomial of $L'$. From the work of Tanaka~\cite{Ta06} and Yokota~\cite{Yo95} it follows that in the case of a reduced alternating knot diagram $D'_n$ of $L'_n$ (as is the case here) this inequality transforms to
\begin{equation*}
\tb(L'_n)\leq\slf(D'_n)-\operatorname{r}(D'_n)
\end{equation*} 
where $\slf(D'_n)$ is the self-linking number of the knot diagram $D'_n$ and $\operatorname{r}(D'_n)$ is the number of regions in the knot diagram $D'_n$. By coloring the complement of an alternating knot diagram according to the rule in Figure~\ref{fig:regions}(i) one gets the regions as the colored areas. (Moreover, in~\cite{Ta06} and~\cite{Yo95} it is shown that this bound is sharp, but this is not important for the argument here.) 
\begin{figure}[htbp] 
\centering
\def\svgwidth{0,85\columnwidth}
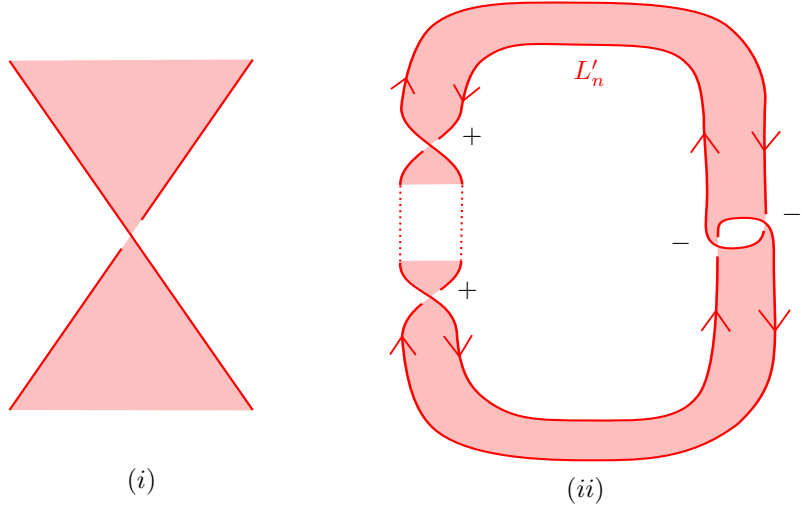
\caption{Computing the Kauffman bound for $L'_n$}
\label{fig:regions}
\end{figure}	

In Figure~\ref{fig:regions}(ii) one can count $\slf(D'_n)=2n-2$ and $\operatorname{r}(D'_n)=2n+1$. Therefore, every Legendrian realization of $L'_n$ in $(S^3,\xi_{st})$ has $\tb(L'_n)\leq-3$. But earlier we computed $\tb_{new}=\tb_{old}=-1$ in the new surgered contact $3$-sphere. Therefore this new contact $3$-sphere has to be overtwisted.

However, we are not done here. We still have to ensure that we can find Legendrian representatives of the Legendrian knots $L$ and $L'$ from Figures~\ref{fig:nsmallerzero} and~\ref{fig:nbiggerzero} sitting completely in the knot exteriors $(S^3\setminus\mathring{\nu U_{\slf}})$. 

It could happen that some tubular neighborhoods $\nu U_{\slf}$ are so large that they hit the Legendrian knots $L$ or $L'$. Then $L$ or $L'$ would not represent a Legendrian knot in the surgered manifold and the above argument would not work. 

In general, this can really happen as we can see as follows. If we assume that we can realize the Legendrian knot $L$ from Figure~\ref{fig:nsmallerzero} in all exteriors of $U_{-1}$, the above proof shows that there exists no cosmetic contact surgery along $U_{-1}$ with surgery coefficient of the form $r=1/n$ for $n<0$. This is a contradiction to the result of Etnyre--Ghrist (Theorem~\ref{thm:etnyre-ghrist}). It follows that the Legendrian knot $L$ from Figure~\ref{fig:nsmallerzero} cannot be realized in all exteriors of $U_{-1}$.

On the other hand, we will show next that for a stabilized transverse unknot $T$ we can always realize the Legendrian knots  $L$ and $L'$ in all of its exteriors $(S^3\setminus\mathring{\nu T},\xist)$.

For that, let $(\nu T,\xist)$ be a standard tubular neighborhood of $T$. We choose an identification of $(\nu T,\xist)$ with
\begin{equation*}
\big(S^1\times D_\varepsilon^2,\ker(d\theta+ r^2\,d\varphi)\big)
\end{equation*}
such that $S^1\times\{\text{pt}\}$ corresponds to the surface longitude $\lambda$ of $T$. 

First, we want to show that $\varepsilon\leq1$. For that, let $k$ be the maximal natural number such that $\varepsilon^2>1/k$. The characteristic foliation on 
\begin{equation*}
\partial\big(S^1\times D^2_{1/\sqrt{k}}\big)\subset\mathring{\nu T}
\end{equation*}
is given by parallel linear curves of slope $\lambda-k\mu$. By a $C^{\infty}$-close isotopy inside $\mathring{\nu T}$ we can transform $\partial(S^1\times D^2_{1/\sqrt{k}})$ to a convex surface with two parallel dividing curves of slope $\lambda-k\mu$~\cite[Example~4.8.10]{Ge08}. This convex surface bounds a solid torus $V$ inside $\mathring{\nu T}$. 

By Honda's classification of tight contact structures on solid tori with prescribed boundaries~\cite[Theorem~2.3]{Ho00} the contact structure on $V$ is unique and its spine $L$ is a Legendrian push-off of $T$ with $\tb=-k$. On the other hand, we know that the classical invariants of a transverse knot $T$ and of its Legendrian push-offs $L$ are related by
\begin{equation*}
\slf(T)=\tb(L)-\rot(L).
\end{equation*}
If we now assume that $\varepsilon$ is larger than $1$, we get $k=1$. By the above discussion it follows that $L$ is a Legendrian unknot with $\tb(L)=-1$. It follows that $\rot(L)=0$ and therefore its transverse push-off $T$ has self-linking number $\slf(T)=-1$ and would not be stabilized. Altogether, we conclude that $\varepsilon\leq1$ for all standard tubular neighborhoods of stabilized Legendrian unknots.

Working backwards, we see that for a transverse unknot with $\slf=-1$ there really exist standard tubular neighborhoods of size $\varepsilon\geq1$.

Next, let $k$ be the natural number such  that
\begin{equation*}
\frac{1}{k}\geq\varepsilon^2>\frac{1}{k+1}.
\end{equation*}
As above we consider the surface $\partial(S^1\times D^2_{1/\sqrt{k}})$ and perturb it to a convex surface $\partial V$ with two parallel dividing curves of slope $\lambda-k\mu$. But this time $\nu T$ is contained in $V$ and and we need to show that we can assume $\partial V$ sitting in the exteriors. For that we take a transverse unknot $U_{-1}$ with $\slf=-1$ and tubular neighborhood $\nu U_{-1}$ of size $\varepsilon\geq1$. We can perform stabilizations inside $\nu U_{-1}$ to obtain the stabilized unknot $T$ such that $\nu T\subset V\subset \nu U_{-1}$. It follows that we can assume $\partial V$ sitting in the exteriors of $T$. Therefore, it is sufficient to realize the Legendrian knot $L$ from Figure~\ref{fig:nsmallerzero} on $\partial V$.

For that, we chose a Legendrian realization $L$ of $\mu$ on $\partial V$~\cite[Theorem~3.7]{Ho00}. From the Bennequin inequality it follows that $\tb(L)\leq1$. By~\cite[Proposition~3.1]{Ho00} we can make a meridional disk $D^2$ with $\partial D^2=\mu$ inside $V$ convex. Since the characteristic foliation on $\partial V$ consists of two parallel curves of slope $\lambda-k\mu$ the characteristic foliation $\Gamma$ on $D^2$ is a single arc with endpoints on $L$, see also~\cite[Proposition~4.3]{Ho00}. The Thurston-Bennequin invariant of $L$ computes as
\begin{equation*}
\tb(L)=-\frac{1}{2} \#(L\cap \Gamma)=-1.
\end{equation*}   
It follows that $L$ is a Legendrian meridian of $T$ with $\tb=-1$ sitting completely in the exteriors of $T$ and therefore representing the Legendrian knot $L$ from Figure~\ref{fig:nsmallerzero}.

It remains to show that we also can realize the Legendrian knot $L'$ from Figure~\ref{fig:nbiggerzero} in all exteriors of $T$. For that, observe that $L'$ is a connected sum (along some band) of two copies of $L$. We take two copies $L_1$ and $L_2$ of $L$ in $(S^3\setminus\mathring{\nu T},\xist)$ as in Figure~\ref{fig:connectedsum}(i) and perform a connected sum away from $\nu T$ as in Figure~\ref{fig:connectedsum}(ii). By applying a few Legendrian Reidemeister moves away from $\nu T$ we see that the resulting Legendrian knot is isotopic to $L'$ in the knot exterior of $T$. \hfill$\square$

\begin{figure}[htbp] 
	\centering
	\def\svgwidth{0,99\columnwidth}
	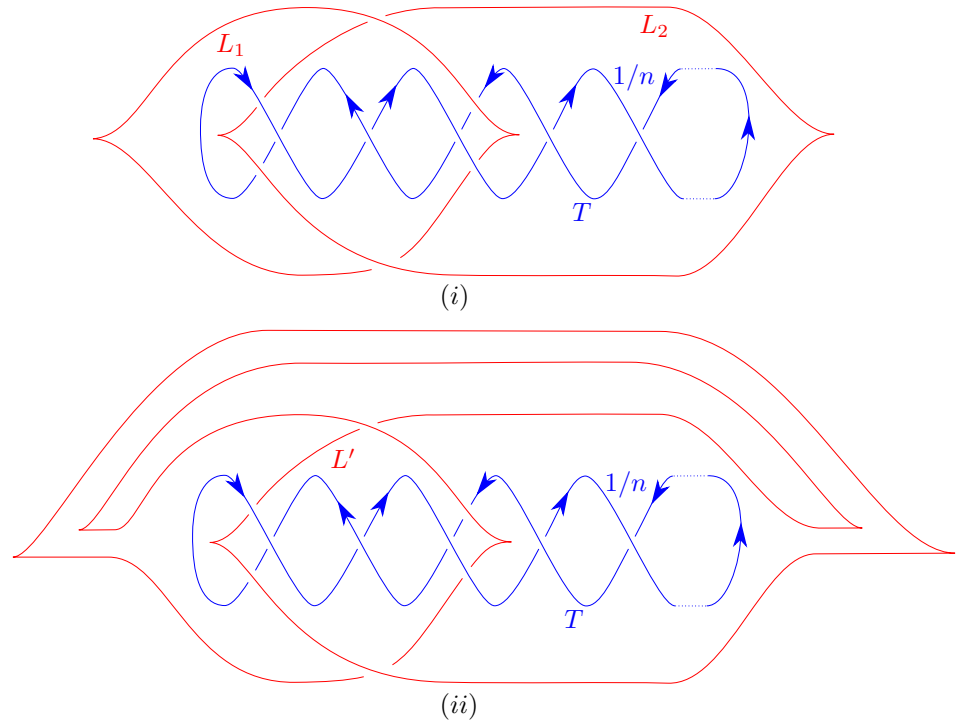
	\caption{Constructing $L'$ as a connected sum of two copies of $L$}
	\label{fig:connectedsum}
\end{figure}


\section{Proof of Corollary~\ref{thm:transverseSurgery2} and~\ref{cor:transverse}}
\label{section:proofKnotExterior} 

The transverse surgery theorem follows now easily from Theorem~\ref{thm:transverseSurgery} and the topological result by Gordon and Luecke~\cite{GoLu89}.

\begin{proof}[Proof of Corollary~\ref{thm:transverseSurgery2}]\hfill\\
	By the theorem of Gordon and Luecke~\cite{GoLu89} $T$ is topological equivalent to the unknot and an easy computation in homology shows that the surgery coefficient has to be of the form $r=1/n$, for $n\in \Z$. Using Theorem~\ref{thm:transverseSurgery} the result follows.	
\end{proof}

The proof of the transverse knot exterior theorem works now very similar as for topological or Legendrian knots. Compare~\cite{Ke16,Ke17}.

\begin{proof}[Proof of Corollary~\ref{cor:transverse}]\hfill\\
Pick a contactomorphism between the exteriors
\begin{equation*}
h\colon\big(S^3\setminus\mathring{\nu T_1},\xi_{st}\big)\longrightarrow \big(S^3\setminus\mathring{\nu T_2},\xi_{st}\big),
\end{equation*}
and then consider the following commutative diagram:\\
\hfill\\
\begin{xy}
(0,40)*+{\big(S^3,\xi_{st}\big)}="a";(8,40)*+{\cong}="a1"; (24,40)*+{\big(S^3_{T_1}  (\mu_1),\xi_{T_1}  (\mu_1)\big)}="b"; (40,40)*+{:=}="b1"; (52,40)*+{\big(S^1\times D^2,\xi_S\big)}="c"; (75,40)*+{+}="d"; (100,40)*+{\big(S^3\setminus\mathring{\nu T_1},\xi_{st}\big)}="e";(114,40)*+{\big/_\sim}="e1";%
(53,35)*+{\mu_0}="f"; (99,35)*+{\mu_1}="g";%
(75,20)*+{\circlearrowright}="k1";%
(53,5)*+{\mu_0}="l"; (92,5)*+{r_2:=h(\mu_1)}="m";%
(24,0)*+{\big(S^3_{T_2}  (r_2),\xi_{T_2}  (r_2)\big)}="n";(40,0)*+{:=}="n1"; (52,0)*+{\big(S^1\times D^2,\xi_S\big)}="o"; (75,0)*+{+}="p"; (100,0)*+{\big(S^3\setminus\mathring{\nu T_2},\xi_{st}\big)}="q";(114,0)*+{\big/_\sim}="q1";%
{\ar@{-->}_{f} "b";"n"};{\ar@{->}@/_1pc/ _{\operatorname{Id}} "c";"o"};{\ar@{->}@/^1pc/^h "e";"q"};%
{\ar@{|->}_{\varphi_1} "f";"g"};{\ar@{|->}^{h\circ\varphi_1} "l";"m"};
\end{xy}\\
\hfill\\
Here the contact solid torus $(S^1\times D^2,\xi_S)$ is chosen to be equal to $(\nu T_1,\xi_{st})$. Because $\operatorname{Id}$ and $h$ on the two factors are contactomorphisms and because they both send the characteristic foliations of the boundaries to each other, these two maps glue together to a contactomorphism $f$ of the whole contact manifolds~\cite[Section~2.5.4]{Ge08}. From the transverse contact Dehn surgery theorem~\ref{cor:transverse} it follows that $r_2$ is equal to $\mu_2$, or $T_2$ is equivalent to the transverse unknot $U_{-1}$ with self-linking number $\slf=-1$.

If $r_2=\mu_2$ then this is a trivial contact Dehn surgery, and so the contactomorphism $f$ maps $T_1$ to $T_2$.

In the other case the same argument with $T_1$ and $T_2$ reversed shows that $T_1$ is also equivalent to the transverse unknot $U_{-1}$. 
\end{proof}


\section{Links in $S^3$ and knots in general $3$-manifolds}
\label{section:open}

\subsection{The transverse link complement problem}\hfill

An obvious question is if the transverse knot exterior theorem is also true for links. It will turn out that this is not the case. There exist non-equivalent transverse $2$-component links in $(S^3,\xist)$ with contactomorphic exteriors. The idea for constructing such examples is the same as for topological links~\cite{Wh37}.

\begin{ex}[Counterexample for transverse links]
Consider the so-called Whitehead link from Figure~\ref{fig:newlink}(i). Chose $T$ in $(S^3,\xist)$ to be a transverse unknot with $\slf=-1$ with standard tubular neighborhood from Theorem~\ref{thm:etnyre-ghrist} of Etnyre-Ghrist. So that a contact $(1/n)$-Dehn surgery along $T$ with respect to this standard tubular neighborhood yields again $(S^3,\xist)$. Next chose a transverse realization $T'$ of $L'$ in the exterior of $T$ (with respect to the chosen tubular neighborhood of $T$).

A contact $(1/n)$-Dehn surgery along $T$ yields $(S^3,\xist)$ back but the transverse knot $T'$ changes to a transverse knot with underlying topological knot type $L'_n$ from Figure~\ref{fig:newlink}(ii) which is clearly non-equivalent to $L'$. Moreover from the proof of Theorem~\ref{thm:etnyre-ghrist} it follows directly that the new glued-in solid torus is the standard tubular neighborhood of a new transverse knot $T_n$. 

Therefore, the transverse links $T\sqcup T'$ and $T_n\sqcup T'_n$ have contactomorphic exteriors but are not equivalent (since they are not even topologically equivalent).

Observe, that we can construct like this infinitely many pairwise non equivalent transverse $2$-component links that have exteriors that are all contactomorphic.
\end{ex}

\subsection{The transverse knot complement problem in general manifolds}\hfill

We can also ask the same question in general contact manifolds: Given two transverse knots $T_1$ and $T_2$ in a closed contact $3$-manifold $(M,\xi)$ with contactomorphic exteriors. Are $T_1$ and $T_2$ equivalent?

In the topological category the following easy construction is working. Take two solid tori and glue them together along their boundaries to obtain the standard genus-$1$ Heegaard splitting of a lens space. The spines of these solid tori have homeomorphic exteriors by construction but it is easy to compute that they are orientation preserving equivalent if and only if $q^2\equiv 1\mod{p}$. In fact, these are the only known examples of knots with homeomorphic exteriors that are not orientation preserving equivalent. See~\cite{Ro93} or~\cite[Section~1.3]{Ke17} for details.

The same construction does not work for transverse knots in contact manifolds.

\begin{ex}[Gluing standard neighborhood of transverse knots]
Consider a standard neighborhood $\nu T$ of a transverse knot $T$ with size $\varepsilon^2={x}/{y}$. The characteristic foliation on the boundary $\partial(\nu T)$ is given by the vector field 
\begin{equation*}
X=\frac{x}{y}\, \partial_\theta-\,\partial_\varphi .
\end{equation*}
If one denotes the longitude $S^1\times \{p\}$ of $\partial(\nu T)$ by $\lambda$, then the characteristic foliation of $\partial(\nu T)$ is given by parallel linear curves of slope $x\lambda-y\mu$.

Next, we take two contactomorphic copies $(V_1,\xi_1)$ and $(V_2,\xi_2)$ of such a standard tubular neighborhood. After a possible change of the longitudes we see that this is the case if and only if the slopes of both boundary characteristic foliations agree, i.e.\ the characteristic foliation is given by linear curves of the form $x\lambda_i-y\mu_i$.

Now we glue the solid tori $V_1$ and $V_2$ together to obtain a lens space $L(p,q)$ as follows
\\
\hfill\\
\begin{xy}
(-14,12)*+{ }="a0";(7,12)*+{L(p,q)}="a5";(17,12)*+{=}="a1"; (25,12)*+{V_1}="c"; (45,12)*+{+}="d"; (65,12)*+{V_2}="e";(80,12)*+{\big/_\sim}="e1";%
(25,6)*+{\mu_1}="f"; (65,6)*+{q\mu_2-p\lambda_2,}="g";%
(25,0)*+{\lambda_1}="h"; (65,0)*+{r\mu_2+s\lambda_2,}="i";%
{\ar@{|->} "f";"g"};
{\ar@{|->} "h";"i"};
\end{xy}\\
\hfill\\
where we require the gluing map to be orientation reversing, i.e.\
\begin{equation*}
\det(A)=\det\left(\begin{matrix}
s & p\\
-r &q
\end{matrix}\right)=-1.
\end{equation*}
The contact structures $\xi_1$ and $\xi_2$ on the solid tori glue together to a global contact structure on the lens space $L(p,q)$ if the gluing map preserves the boundary characteristic foliations. This is the case if and only if
\begin{equation*}
A\left(\begin{matrix}
x\\
y
\end{matrix}\right)=\left(\begin{matrix}
x\\
y
\end{matrix}\right).
\end{equation*}
Every gluing map that satisfies the above condition yields a contact structure on a lens space, such that the transverse spines of its standard genus-$1$ Heegaard splitting have contactomorphic exteriors. However, in all these examples the transverse knots are actually equivalent, as we can see as follows.

From the above equations, we conclude that the eigenvalues of $A$ have to be $+1$ and $-1$ and therefore
\begin{equation*}
 0=\operatorname{trace}(A)=q+s.
\end{equation*}
It follows that $q^2=1-pr\equiv 1 \mod{(p)}.$ But under this condition we can construct a diffeomorphism that interchanges both solid torus. In fact, it is easy to compute that this diffeomorphism is also a contactomorphism that maps on transverse spine to the other.
\end{ex}

It remains open if we can find counterexamples  to Corollary~\ref{cor:transverse} in general contact manifolds. It would also be interesting to know what happens in overtwisted contact $3$-spheres.

\begin{ques}[Transverse knot exterior problem in overtwisted $3$-spheres]
Let $T_1$ and $T_2$ be knots in an overtwisted contact $3$-sphere $(S^3,\xiot)$ with contactomorphic complements $(S^3\setminus T_i,\xiot)$. Is $T_1$ equivalent to $T_2$?
\end{ques}

Recall, that in overtwisted $3$-spheres coarse-equivalence is a weaker condition than isotopy of transverse links~\cite{Vo16}.

\subsection{Complements and exteriors}\hfill

As already mentioned in the introduction from the work of Edwards~\cite{Ed64} it follows that for topological knots two knot complements are homeomorphic if and only if the knot exteriors are homeomorphic. For transverse knots in contact manifold this is not clear (see~\cite[Section~6.4]{Ke17} for more discussion on this problem). In particular, the knot complement problem for transverse knots in $(S^3,\xist)$ remains open.

\begin{ques}[Transverse knot complement problem]
	Let $T_1$ and $T_2$ be knots in $(S^3,\xist)$ with contactomorphic complements $(S^3\setminus T_i,\xist)$. Is $T_1$ equivalent to $T_2$?
\end{ques}





\end{document}

%% file: classificationOfTransverseUnknots.pdf_tex
\begingroup%
  \makeatletter%
  \providecommand\color[2][]{%
    \errmessage{(Inkscape) Color is used for the text in Inkscape, but the package 'color.sty' is not loaded}%
    \renewcommand\color[2][]{}%
  }%
  \providecommand\transparent[1]{%
    \errmessage{(Inkscape) Transparency is used (non-zero) for the text in Inkscape, but the package 'transparent.sty' is not loaded}%
    \renewcommand\transparent[1]{}%
  }%
  \providecommand\rotatebox[2]{#2}%
  \ifx\svgwidth\undefined%
    \setlength{\unitlength}{508.95792975bp}%
    \ifx\svgscale\undefined%
      \relax%
    \else%
      \setlength{\unitlength}{\unitlength * \real{\svgscale}}%
    \fi%
  \else%
    \setlength{\unitlength}{\svgwidth}%
  \fi%
  \global\let\svgwidth\undefined%
  \global\let\svgscale\undefined%
  \makeatother%
  \begin{picture}(1,0.18564873)%
    \put(0,0){\includegraphics[width=\unitlength,page=1]{classificationOfTransverseUnknots.pdf}}%
    \put(0.02955703,0.03595357){\color[rgb]{0,0,0.98039216}\makebox(0,0)[lt]{\begin{minipage}{0.14707923\unitlength}\raggedright $\slf=-1$\end{minipage}}}%
    \put(0,0){\includegraphics[width=\unitlength,page=2]{classificationOfTransverseUnknots.pdf}}%
    \put(0.26981165,0.03604826){\color[rgb]{0,0,0.98039216}\makebox(0,0)[lt]{\begin{minipage}{0.14707923\unitlength}\raggedright $\slf=-3$\end{minipage}}}%
    \put(0.64593031,0.03604827){\color[rgb]{0,0,0.98039216}\makebox(0,0)[lt]{\begin{minipage}{0.14707923\unitlength}\raggedright $\slf=-5$\end{minipage}}}%
  \end{picture}%
\endgroup%

%% file: crossings.pdf_tex
\begingroup%
  \makeatletter%
  \providecommand\color[2][]{%
    \errmessage{(Inkscape) Color is used for the text in Inkscape, but the package 'color.sty' is not loaded}%
    \renewcommand\color[2][]{}%
  }%
  \providecommand\transparent[1]{%
    \errmessage{(Inkscape) Transparency is used (non-zero) for the text in Inkscape, but the package 'transparent.sty' is not loaded}%
    \renewcommand\transparent[1]{}%
  }%
  \providecommand\rotatebox[2]{#2}%
  \ifx\svgwidth\undefined%
    \setlength{\unitlength}{399.08920407bp}%
    \ifx\svgscale\undefined%
      \relax%
    \else%
      \setlength{\unitlength}{\unitlength * \real{\svgscale}}%
    \fi%
  \else%
    \setlength{\unitlength}{\svgwidth}%
  \fi%
  \global\let\svgwidth\undefined%
  \global\let\svgscale\undefined%
  \makeatother%
  \begin{picture}(1,0.47240456)%
    \put(0,0){\includegraphics[width=\unitlength,page=1]{crossings.pdf}}%
    \put(0.14566993,0.30621478){\color[rgb]{1,0,0}\makebox(0,0)[lt]{\begin{minipage}{0.06514834\unitlength}\raggedright $L$\end{minipage}}}%
    \put(0.40077522,0.30354561){\color[rgb]{1,0,0}\makebox(0,0)[lt]{\begin{minipage}{0.07015975\unitlength}\raggedright $L$\end{minipage}}}%
    \put(0.53508104,0.30555017){\color[rgb]{1,0,0}\makebox(0,0)[lt]{\begin{minipage}{0.05913465\unitlength}\raggedright $L$\end{minipage}}}%
    \put(0.80068582,0.30454788){\color[rgb]{1,0,0}\makebox(0,0)[lt]{\begin{minipage}{0.05713008\unitlength}\raggedright $L$\end{minipage}}}%
    \put(0.00687832,0.30655245){\color[rgb]{0.02352941,0,0.98823529}\makebox(0,0)[lt]{\begin{minipage}{0.06715291\unitlength}\raggedright $T$\end{minipage}}}%
    \put(0.26844461,0.30254332){\color[rgb]{0.02352941,0,0.98823529}\makebox(0,0)[lt]{\begin{minipage}{0.06715291\unitlength}\raggedright $T$\end{minipage}}}%
    \put(0.6793803,0.30555017){\color[rgb]{0.02352941,0,0.98823529}\makebox(0,0)[lt]{\begin{minipage}{0.06715291\unitlength}\raggedright $T$\end{minipage}}}%
    \put(0.95300334,0.30755473){\color[rgb]{0.02352941,0,0.98823529}\makebox(0,0)[lt]{\begin{minipage}{0.06715291\unitlength}\raggedright $T$\end{minipage}}}%
    \put(0.05996991,0.47894498){\color[rgb]{0.02352941,0,0}\makebox(0,0)[lt]{\begin{minipage}{0.06715291\unitlength}\raggedright $(1)$\end{minipage}}}%
    \put(0.32557469,0.47994727){\color[rgb]{0.02352941,0,0}\makebox(0,0)[lt]{\begin{minipage}{0.06715291\unitlength}\raggedright $(2)$\end{minipage}}}%
    \put(0.61423196,0.47994727){\color[rgb]{0.02352941,0,0}\makebox(0,0)[lt]{\begin{minipage}{0.06715291\unitlength}\raggedright $(3)$\end{minipage}}}%
    \put(0.87482533,0.4779427){\color[rgb]{0.02352941,0,0}\makebox(0,0)[lt]{\begin{minipage}{0.06715291\unitlength}\raggedright $(4)$\end{minipage}}}%
    \put(0,0){\includegraphics[width=\unitlength,page=2]{crossings.pdf}}%
    \put(0.39897882,0.06664541){\color[rgb]{1,0,0}\makebox(0,0)[lt]{\begin{minipage}{0.07015975\unitlength}\raggedright $L$\end{minipage}}}%
    \put(0.79888942,0.06764769){\color[rgb]{1,0,0}\makebox(0,0)[lt]{\begin{minipage}{0.05713008\unitlength}\raggedright $L$\end{minipage}}}%
    \put(0.26664821,0.06564313){\color[rgb]{0.02352941,0,0.98823529}\makebox(0,0)[lt]{\begin{minipage}{0.06715291\unitlength}\raggedright $T$\end{minipage}}}%
    \put(0.95120695,0.07065454){\color[rgb]{0.02352941,0,0.98823529}\makebox(0,0)[lt]{\begin{minipage}{0.06715291\unitlength}\raggedright $T$\end{minipage}}}%
    \put(0.07511181,0.11093876){\color[rgb]{0,0,0}\makebox(0,0)[lt]{\begin{minipage}{0.03808672\unitlength}\raggedright ?\end{minipage}}}%
    \put(0.62837158,0.11093876){\color[rgb]{0,0,0}\makebox(0,0)[lt]{\begin{minipage}{0.03808672\unitlength}\raggedright ?\end{minipage}}}%
  \end{picture}%
\endgroup%

%% file: nsmallerzero.pdf_tex
\begingroup%
  \makeatletter%
  \providecommand\color[2][]{%
    \errmessage{(Inkscape) Color is used for the text in Inkscape, but the package 'color.sty' is not loaded}%
    \renewcommand\color[2][]{}%
  }%
  \providecommand\transparent[1]{%
    \errmessage{(Inkscape) Transparency is used (non-zero) for the text in Inkscape, but the package 'transparent.sty' is not loaded}%
    \renewcommand\transparent[1]{}%
  }%
  \providecommand\rotatebox[2]{#2}%
  \ifx\svgwidth\undefined%
    \setlength{\unitlength}{526.40726872bp}%
    \ifx\svgscale\undefined%
      \relax%
    \else%
      \setlength{\unitlength}{\unitlength * \real{\svgscale}}%
    \fi%
  \else%
    \setlength{\unitlength}{\svgwidth}%
  \fi%
  \global\let\svgwidth\undefined%
  \global\let\svgscale\undefined%
  \makeatother%
  \begin{picture}(1,0.40349514)%
    \put(0,0){\includegraphics[width=\unitlength,page=1]{nsmallerzero.pdf}}%
    \put(0.49940065,0.14086413){\color[rgb]{0,0,1}\makebox(0,0)[lt]{\begin{minipage}{0.29606309\unitlength}\raggedright $U_{\operatorname{sl}}$\end{minipage}}}%
    \put(0,0){\includegraphics[width=\unitlength,page=2]{nsmallerzero.pdf}}%
    \put(0.90093075,0.17559508){\color[rgb]{0.98823529,0,0}\makebox(0,0)[lt]{\begin{minipage}{0.03846786\unitlength}\raggedright $L$\end{minipage}}}%
    \put(0.65212514,0.31569333){\color[rgb]{0.01568627,0,0.98431373}\makebox(0,0)[lt]{\begin{minipage}{0.07285555\unitlength}\raggedright $1/n$\end{minipage}}}%
    \put(0.21309862,0.33865472){\color[rgb]{0.01568627,0,0}\makebox(0,0)[lt]{\begin{minipage}{0.07285555\unitlength}\raggedright $(4)$\end{minipage}}}%
    \put(0.2078343,0.12711557){\color[rgb]{0.01568627,0,0}\makebox(0,0)[lt]{\begin{minipage}{0.07285555\unitlength}\raggedright $(2)$\end{minipage}}}%
    \put(0,0){\includegraphics[width=\unitlength,page=3]{nsmallerzero.pdf}}%
  \end{picture}%
\endgroup%

%% file: nbiggerzero.pdf_tex
\begingroup%
  \makeatletter%
  \providecommand\color[2][]{%
    \errmessage{(Inkscape) Color is used for the text in Inkscape, but the package 'color.sty' is not loaded}%
    \renewcommand\color[2][]{}%
  }%
  \providecommand\transparent[1]{%
    \errmessage{(Inkscape) Transparency is used (non-zero) for the text in Inkscape, but the package 'transparent.sty' is not loaded}%
    \renewcommand\transparent[1]{}%
  }%
  \providecommand\rotatebox[2]{#2}%
  \ifx\svgwidth\undefined%
    \setlength{\unitlength}{533.53327572bp}%
    \ifx\svgscale\undefined%
      \relax%
    \else%
      \setlength{\unitlength}{\unitlength * \real{\svgscale}}%
    \fi%
  \else%
    \setlength{\unitlength}{\svgwidth}%
  \fi%
  \global\let\svgwidth\undefined%
  \global\let\svgscale\undefined%
  \makeatother%
  \begin{picture}(1,0.39222016)%
    \put(0.67091103,0.07167952){\color[rgb]{0,0,0.98823529}\makebox(0,0)[lt]{\begin{minipage}{0.11320862\unitlength}\raggedright $U_{\slf}$\end{minipage}}}%
    \put(0.2700066,0.38395575){\color[rgb]{0.98823529,0,0}\makebox(0,0)[lt]{\begin{minipage}{0.04887286\unitlength}\raggedright $L'$\end{minipage}}}%
    \put(0.74348436,0.32144508){\color[rgb]{0.01568627,0,1}\makebox(0,0)[lt]{\begin{minipage}{0.07188247\unitlength}\raggedright $1/n$\end{minipage}}}%
    \put(0,0){\includegraphics[width=\unitlength,page=1]{nbiggerzero.pdf}}%
  \end{picture}%
\endgroup%

%% file: newlink.pdf_tex
\begingroup%
  \makeatletter%
  \providecommand\color[2][]{%
    \errmessage{(Inkscape) Color is used for the text in Inkscape, but the package 'color.sty' is not loaded}%
    \renewcommand\color[2][]{}%
  }%
  \providecommand\transparent[1]{%
    \errmessage{(Inkscape) Transparency is used (non-zero) for the text in Inkscape, but the package 'transparent.sty' is not loaded}%
    \renewcommand\transparent[1]{}%
  }%
  \providecommand\rotatebox[2]{#2}%
  \ifx\svgwidth\undefined%
    \setlength{\unitlength}{416.750293bp}%
    \ifx\svgscale\undefined%
      \relax%
    \else%
      \setlength{\unitlength}{\unitlength * \real{\svgscale}}%
    \fi%
  \else%
    \setlength{\unitlength}{\svgwidth}%
  \fi%
  \global\let\svgwidth\undefined%
  \global\let\svgscale\undefined%
  \makeatother%
  \begin{picture}(1,0.49170062)%
    \put(0,0){\includegraphics[width=\unitlength,page=1]{newlink.pdf}}%
    \put(0.22021567,0.44294228){\color[rgb]{0.98823529,0,0}\makebox(0,0)[lt]{\begin{minipage}{0.05537655\unitlength}\raggedright $L'$\end{minipage}}}%
    \put(0.13174348,0.23351093){\color[rgb]{0,0,0.99607843}\makebox(0,0)[lt]{\begin{minipage}{0.06273822\unitlength}\raggedright $T$\end{minipage}}}%
    \put(0.14063248,0.33190966){\color[rgb]{0.01568627,0,0.97647059}\makebox(0,0)[lt]{\begin{minipage}{0.09202559\unitlength}\raggedright $1/n$\end{minipage}}}%
    \put(0.8101509,0.43835393){\color[rgb]{0.98823529,0,0}\makebox(0,0)[lt]{\begin{minipage}{0.08737175\unitlength}\raggedright $L'_n$\end{minipage}}}%
    \put(0.6589008,0.28128653){\color[rgb]{0.98823529,0,0}\makebox(0,0)[lt]{\begin{minipage}{0.06119386\unitlength}\raggedright $-n$\end{minipage}}}%
    \put(0.48642739,0.28322563){\color[rgb]{0.01568627,0,0}\makebox(0,0)[lt]{\begin{minipage}{0.1055993\unitlength}\raggedright $\cong$\end{minipage}}}%
    \put(0.20428781,0.03405081){\color[rgb]{0.01568627,0,0}\makebox(0,0)[lt]{\begin{minipage}{0.09202559\unitlength}\raggedright $(i)$\end{minipage}}}%
    \put(0.80928816,0.03502036){\color[rgb]{0.01568627,0,0}\makebox(0,0)[lt]{\begin{minipage}{0.09202559\unitlength}\raggedright $(ii)$\end{minipage}}}%
  \end{picture}%
\endgroup%

%% file: regions.pdf_tex
\begingroup%
  \makeatletter%
  \providecommand\color[2][]{%
    \errmessage{(Inkscape) Color is used for the text in Inkscape, but the package 'color.sty' is not loaded}%
    \renewcommand\color[2][]{}%
  }%
  \providecommand\transparent[1]{%
    \errmessage{(Inkscape) Transparency is used (non-zero) for the text in Inkscape, but the package 'transparent.sty' is not loaded}%
    \renewcommand\transparent[1]{}%
  }%
  \providecommand\rotatebox[2]{#2}%
  \ifx\svgwidth\undefined%
    \setlength{\unitlength}{348.41039187bp}%
    \ifx\svgscale\undefined%
      \relax%
    \else%
      \setlength{\unitlength}{\unitlength * \real{\svgscale}}%
    \fi%
  \else%
    \setlength{\unitlength}{\svgwidth}%
  \fi%
  \global\let\svgwidth\undefined%
  \global\let\svgscale\undefined%
  \makeatother%
  \begin{picture}(1,0.6286525)%
    \put(0,0){\includegraphics[width=\unitlength,page=1]{regions.pdf}}%
    \put(0.05266186,-0.0003624){\color[rgb]{0,0,0}\makebox(0,0)[lt]{\begin{minipage}{0.17415876\unitlength}\raggedright \end{minipage}}}%
    \put(0.14605134,0.05138039){\color[rgb]{0,0,0}\makebox(0,0)[lt]{\begin{minipage}{0.08076925\unitlength}\raggedright $(i)$\end{minipage}}}%
    \put(0.68925046,0.04072853){\color[rgb]{0,0,0}\makebox(0,0)[lt]{\begin{minipage}{0.0959135\unitlength}\raggedright $(ii)$\end{minipage}}}%
    \put(0.56052445,0.47233933){\color[rgb]{0,0,0}\makebox(0,0)[lt]{\begin{minipage}{0.08076925\unitlength}\raggedright $+$\end{minipage}}}%
    \put(0.55446304,0.2819271){\color[rgb]{0,0,0}\makebox(0,0)[lt]{\begin{minipage}{0.08076925\unitlength}\raggedright $+$\end{minipage}}}%
    \put(0.95578538,0.37657861){\color[rgb]{0,0,0}\makebox(0,0)[lt]{\begin{minipage}{0.08076925\unitlength}\raggedright $-$\end{minipage}}}%
    \put(0.81822519,0.33998002){\color[rgb]{0,0,0}\makebox(0,0)[lt]{\begin{minipage}{0.08076925\unitlength}\raggedright $-$\end{minipage}}}%
    \put(0.69833331,0.55704744){\color[rgb]{0.97647059,0,0}\makebox(0,0)[lt]{\begin{minipage}{0.10096158\unitlength}\raggedright $L'_n$\end{minipage}}}%
  \end{picture}%
\endgroup%

%% file: connectedsum.pdf_tex
\begingroup%
  \makeatletter%
  \providecommand\color[2][]{%
    \errmessage{(Inkscape) Color is used for the text in Inkscape, but the package 'color.sty' is not loaded}%
    \renewcommand\color[2][]{}%
  }%
  \providecommand\transparent[1]{%
    \errmessage{(Inkscape) Transparency is used (non-zero) for the text in Inkscape, but the package 'transparent.sty' is not loaded}%
    \renewcommand\transparent[1]{}%
  }%
  \providecommand\rotatebox[2]{#2}%
  \ifx\svgwidth\undefined%
    \setlength{\unitlength}{905.55348565bp}%
    \ifx\svgscale\undefined%
      \relax%
    \else%
      \setlength{\unitlength}{\unitlength * \real{\svgscale}}%
    \fi%
  \else%
    \setlength{\unitlength}{\svgwidth}%
  \fi%
  \global\let\svgwidth\undefined%
  \global\let\svgscale\undefined%
  \makeatother%
  \begin{picture}(1,0.74540583)%
    \put(0.59392308,0.53706807){\color[rgb]{0,0,0.98823529}\makebox(0,0)[lt]{\begin{minipage}{0.02315288\unitlength}\raggedright $T$\end{minipage}}}%
    \put(0.21565273,0.71612131){\color[rgb]{0.98823529,0,0}\makebox(0,0)[lt]{\begin{minipage}{0.03789268\unitlength}\raggedright $L_1$\end{minipage}}}%
    \put(0.63668178,0.68422474){\color[rgb]{0.01568627,0,1}\makebox(0,0)[lt]{\begin{minipage}{0.04235166\unitlength}\raggedright $1/n$\end{minipage}}}%
    \put(0,0){\includegraphics[width=\unitlength,page=1]{connectedsum.pdf}}%
    \put(0.58554234,0.10460913){\color[rgb]{0,0,0.98823529}\makebox(0,0)[lt]{\begin{minipage}{0.02315288\unitlength}\raggedright $T$\end{minipage}}}%
    \put(0.33696967,0.27622766){\color[rgb]{0.98823529,0,0}\makebox(0,0)[lt]{\begin{minipage}{0.02879487\unitlength}\raggedright $L'$\end{minipage}}}%
    \put(0.62830104,0.25176581){\color[rgb]{0.01568627,0,1}\makebox(0,0)[lt]{\begin{minipage}{0.04235166\unitlength}\raggedright $1/n$\end{minipage}}}%
    \put(0,0){\includegraphics[width=\unitlength,page=2]{connectedsum.pdf}}%
    \put(0.66517988,0.7366782){\color[rgb]{0.98823529,0,0}\makebox(0,0)[lt]{\begin{minipage}{0.03929236\unitlength}\raggedright $L_2$\end{minipage}}}%
    \put(0.45313093,0.45018366){\color[rgb]{0,0,0}\makebox(0,0)[lt]{\begin{minipage}{0.02879487\unitlength}\raggedright $(i)$\end{minipage}}}%
    \put(0.45313093,0.01567073){\color[rgb]{0,0,0}\makebox(0,0)[lt]{\begin{minipage}{0.0336937\unitlength}\raggedright $(ii)$\end{minipage}}}%
  \end{picture}%
\endgroup%